\documentclass{amsart}

\usepackage{amsmath,amssymb,amsfonts,enumerate,amsthm,graphicx}

\renewcommand{\dim}{\mbox{dim}\,}

\newcommand{\lk}{\mbox{lk}\,}

\newtheorem{thm}{Theorem}[section]
\newtheorem{cor}[thm]{Corollary}
\newtheorem{lem}[thm]{Lemma}
\newtheorem{prop}[thm]{Proposition}
\newtheorem{defn}[thm]{Definition}

\newtheorem{exam}[thm]{Example}
\newtheorem{rem}[thm]{Remark}

\numberwithin{equation}{section}

\begin{document}
 \bibliographystyle{amsplain}
 \title{A generalization of $k$-Cohen-Macaulay complexes}

\author{Hassan Haghighi}
\address{Hassan Haghighi\\Department of Mathematics, K. N. Toosi
     University of Technology, Tehran, Iran.}

\author{Siamak Yassemi}
\address{Siamak Yassemi\\School of Mathematics, Statistics \&
Computer Science, University of
Tehran, Tehran Iran.}

\author{Rahim Zaare-Nahandi}
      \address{Rahim Zaare-Nahandi\\School of Mathematics, Statistics \&
      Computer Science, University of Tehran, Tehran, Iran.}

 \thanks{Emails: haghighi@kntu.ac.ir, yassemi@ipm.ir, rahimzn@ut.ac.ir}



\begin{abstract}

\noindent For a positive integer $k$ and a non-negative integer $t$
a class of simplicial complexes, to be denoted by $k$-${\rm CM}_t$,
is introduced. This class generalizes two notions for simplicial
complexes: being $k$-Cohen-Macaulay and $k$-Buchsbaum. In analogy
with the Cohen-Macaulay and Buchsbaum complexes, we give some
characterizations of ${\rm CM}_t(=$1-${\rm CM}_t)$ complexes, in
terms of vanishing of some homologies of its links and, in terms of
vanishing of some relative singular homologies of the geometric
realization of the complex and its punctured space. We show that a
complex is $k$-${\rm CM}_t$ if and only if the links of its nonempty
faces are $k$-${\rm CM}_{t-1}$. We prove that for an integer $s\le
d$, the $(d-s-1)$-skeleton of a $(d-1)$-dimensional $k$-${\rm CM}_t$
complex is $(k+s)$-${\rm CM}_t$. This result generalizes Hibi's
result for Cohen-Macaulay complexes and Miyazaki's result for
Buchsbaum complexes.
\end{abstract}
\maketitle

\section{Introduction}

Let $K$ be a fixed field. The Stanley-Reisner ring of a simplicial
complex over $K$ provides a "bridge" to transfer properties in
commutative algebra such as being Cohen-Macaulay or Buchsbaum into
simplicial complexes. The main advantage in the study of simplicial
complexes is the interplay between their algebraic, combinatorial,
homological and topological properties. Stanley's book \cite{St-95}
is a suitable  reference for a comprehensive introduction to the
subject. The aim of this paper is to introduce and develop basic
properties of a new class of simplicial complexes, called $k$-${\rm
CM}_t$ complexes, which generalizes two notions for simplicial
complexes: being $k$-Cohen-Macaulay, and being $k$-Buchsbaum.

In Section 2, we introduce ${\rm CM}_t$ complexes and discuss their
basic properties. We show that for a pure simplicial complex
$\Delta$ of dimension $(d-1)$ the following are equivalent, (see
Theorems \ref{1sec1} and \ref{Munkres}):
\begin{itemize}
\item[(i)] $\Delta$ is ${\rm CM}_t$.
\item[(ii)] $\tilde{H}_i(\lk_\Delta(\sigma);K)=0$ for all $\sigma\in\Delta$ with $\#\sigma\ge t$ and $i<d-\#\sigma-1$.
\item[(iii)] $H_i(|\Delta|,|\Delta|\setminus p;K)=0$ for all $p\in|\Delta|\setminus|\Delta_{t-2}|$ and all $i<d-1$,
where $|\Delta|$ is the geometric realization of $\Delta$ and
$\Delta_{t-2}$ is the $(t-2)$-skeleton of $\Delta$.
\end{itemize}

In Section 3, $k$-${\rm CM}_t$ complexes are introduced and some of
their basic properties are studied. We show that a complex is
$k$-${\rm CM}_t$ if and only if the links of its nonempty faces are
$k$-${\rm CM}_{t-1}$ (see Proposition \ref{link1}). We consider a
simplicial complex $\Delta$ and certain faces
$\sigma_1,\cdots,\sigma_\ell$ of $\Delta$ such that
\begin{itemize}
\item[(i)] $\sigma_i\cup\sigma_j\notin\Delta$ if $i\neq j$.
\item[(ii)] If $\Delta_1=\{\tau\in\Delta|\tau\nsupseteq\sigma_i\,\,\mbox{for all $i$}\,\}$ then $\dim\Delta_1<\dim\Delta$.
\end{itemize}
In \cite{Hi} Hibi showed that that $\Delta_1$ is 2-Cohen-Macaulay of
dimension $(\dim\Delta-1)$ provided that $\Delta$ is Cohen-Macaulay
and  $\lk_\Delta(\sigma_i)$ is 2-Cohen-Macaulay for all $i$. In
\cite{Mi} Miyazaki extended this result for Buchsbaumness by showing
that if $\Delta$ is a Buchsbaum complex of dimension $d-1$, and
$\lk_\Delta(\sigma_i)$ is 2-Cohen-Macaulay for all $i$, then
$\Delta_1$ is 2-Buchsbaum. We prove that a  similar result is valid
for ${\rm CM}_t$ complexes (see Theorem \ref{delta1}). This leads to
prove that for an integer $s\le d$, the $(d-s-1)$-skeleton of a
$(d-1)$-dimensional $k$-${\rm CM}_t$ complex is $(k+s)$-${\rm CM}_t$
(see Corollary \ref{3sec3}). This generalizes a result of Terai and
Hibi \cite{TH} (also see \cite{Ba}) which asserts that the
1-skeleton of of a simplicial $(d-1)$-sphere with $d\ge 2$ is
$d$-connected. It also generalizes a result of Hibi \cite{Hi} (see
\cite[Introduction]{Mi}) which says that if $\Delta$ is a
Cohen-Macaulay complex of dimension $d-1$, then the $(d-2)$-skeleton
of $\Delta$ is 2-Cohen-Macaulay.

\section{The ${\rm CM}_t$ simplicial complexes}

In this section we introduce ${\rm CM}_t$ complexes and discuss
their basic properties. We give some characterizations of ${\rm
CM}_t$ complexes, in terms of vanishing of some homologies of its
links (see Theorem \ref{1sec1}), and, in terms of vanishing of some
relative singular homologies of the geometric realization of the
complex and its punctured space (see Theorem \ref{Munkres}). First recall that for any face $\sigma$ of the simplicial complex $\Delta$, the link of $\sigma$ is defined as follows:
$$\lk_\Delta(\sigma)=\{\tau\in\Delta|\tau\cup\sigma\in\Delta, \tau\cap\sigma=\varnothing\}.$$

\begin{defn}
Let $K$ be a field, $\Delta$ a simplicial complex of dimension
$(d-1)$ over $K$. Let $t$ be an integer $0\le t\le d-1$. Then
$\Delta$ is called ${\rm CM}_t$ over $K$ if $\Delta$ is pure and
$\lk_\Delta(\sigma)$ is Cohen-Macaulay over $K$ for any
$\sigma\in\Delta$ with $\#\sigma\ge t$.
\end{defn}

We will adopt the convention that for $t\le 0$, ${\rm CM}_t$ means
${\rm CM}_0$. Note that from the results by Reisner \cite{Re} and
Schenzel \cite{Sc} it follows that ${\rm CM}_{0}$ is the same as
Cohen-Macaulayness and ${\rm CM}_1$ is identical with Buchsbaum
property. It is also clear that for any $j\ge i$, ${\rm CM}_i$
implies ${\rm CM}_j$.

\begin{exam}\label{ex}

Let $\Delta$ be the union of two $(d-1)$-simplicies that intersect
in a $(t-2)$-dimensional face where $1\le t \le d-1$. Then $\Delta$
is a ${\rm CM}_t$ complex which is not a ${\rm CM}_{t-1}$ complex.
In fact, if $\Gamma$ is a finite union of $(d-1)$-simplicies where
any two of them intersect in a face of dimension at most $t-2$, then
$\Gamma$ is a ${\rm CM}_t$ complex, and if at least two of the
simplicies have a $(t-2)$-dimensional face in common, then $\Gamma$
is not ${\rm CM}_{t-1}$. These include simplicial complexes
corresponding to the transversal monomial ideals which happen to
have linear resolutions \cite{Z-Z}.

\end{exam}

It is known that the links of a Cohen-Macaulay simplicial complex
are also Cohen-Macaulay, see \cite{Ho}. As the first result of this
section we show that a similar property holds for ${\rm CM}_t$
complexes. In the rest of this paper we freely use the following fact:\\
For all $\sigma\in\Delta$ and all $\tau\in\lk_\Delta(\sigma)$,
$$\lk_{\lk_\Delta(\sigma)}(\tau)=\lk_\Delta(\sigma\cup\tau).$$

\begin{lem}\label{link-link}
Let $\Delta$ be a simplicial complex. Then the following are
equivalent.
\begin{itemize}

\item[(i)] $\Delta$ is a ${\rm CM}_t$ complex.

\item[(ii)] $\Delta$ is pure and $\lk_\Delta(\{x\})$ is ${\rm CM}_{t-1}$ for all $\{x\}\in\Delta$.

\end{itemize}

\end{lem}

\begin{proof}

(i)$\Rightarrow$(ii). Let $\{x\}\in\Delta$ and
$\tau\in\lk_\Delta(\{x\})$ with $\#\tau\ge t-1$. Since $\Delta$ is
${\rm CM}_t$ and $\#(\{x\}\cup\tau)\ge t$ we see that
$\lk_{\lk_\Delta(\{x\})}(\tau)=\lk_\Delta(\{x\}\cup\tau)$ is
Cohen-Macaulay. In addition, since $\Delta$ is pure it follows that
$\lk_\Delta(\{x\})$ is pure for all $x\in\Delta$.

(ii)$\Rightarrow$(i). Let $\sigma\in\Delta$ with $\#\sigma\ge t$.
Let $x\in\sigma$, $\tau=\sigma\setminus\{x\}$. Then $\#\tau\ge t-1$
and
$\lk_\Delta\sigma=\lk_\Delta(\{x\}\cup\tau)=\lk_{\lk_\Delta(\{x\})}(\tau)$
is Cohen-Macaulay.

\end{proof}

We recall Reisner's characterization of Cohen-Macaulay simplicial complexes \cite[Theorem 1]{Re}.

\begin{thm}

Let $\Delta$ be a simplicial complex of dimension $(d-1)$. Then the following are equivalent:

\begin{itemize}

\item[(i)] $\Delta$ is Cohen-Macaulay over $K$.

\item[(ii)] $\tilde{H}_i(\lk_\Delta(\sigma);K)=0$ for any $\sigma\in\Delta$ and all $i<\dim(\lk_\Delta(\sigma))$.

\end{itemize}

\end{thm}

In analogy with the above result, the following theorem provides
equivalent conditions for ${\rm CM}_t$ complexes.

\begin{thm}\label{Reisner}

Let $\Delta$ be a simplicial complex of dimension $(d-1)$. Then the
following are equivalent:

\begin{itemize}\label{1sec1}

\item[(i)] $\Delta$ is ${\rm CM}_t$ over $K$;

\item[(ii)] $\Delta$ is pure and $\tilde{H}_i(\lk_\Delta(\sigma);K)=0$ for all $\sigma\in\Delta$ with $\#\sigma\ge t$ and $i<d-\#\sigma-1$.

\end{itemize}

\end{thm}

\begin{proof}

(i)$\Rightarrow$(ii). Suppose that $\Delta$ is ${\rm CM}_t$ over
$K$. Then $\Delta$ is pure and $\lk_\Delta(\sigma)$ is
Cohen-Macaulay for all $\sigma\in\Delta$ with $\#\sigma\ge t$.
Therefore, $\tilde{H}_i(\lk_{\lk_\Delta(\sigma)}(\tau);K)=0$ for all
$\tau\in\lk_\Delta(\sigma)$ and all
$i<\dim(\lk_{\lk_\Delta(\sigma)}(\tau))$. In particular, for $\tau =
\varnothing$, $\lk_{\lk_\Delta(\sigma)}(\varnothing)=
\lk_\Delta(\sigma)$ and we have
$\tilde{H}_i(\lk_\Delta(\sigma);K)=0$ for all
$i<\dim(\lk_\Delta(\sigma))\le d-\#\sigma-1$.

(ii)$\Rightarrow$(i). We use induction on $t$. Use \cite[Theorem
3.2]{Sc} for the case $t=1$. Assume that the assertion holds for
$t-1$. Let $\{x\}\in\Delta$, $\tau\in\lk_\Delta\{x\}$ with
$\#\tau\ge t-1$. Then by purity of $\Delta$,
$\dim\lk_\Delta\{x\}=d-2$. But $\tau\cup\{x\}\in\Delta$ and hence by
(ii), $\tilde{H}_i(\lk_\Delta(\tau\cup\{x\});K)=0$ for all
$i<d-\#\tau-2$. This implies that
$\tilde{H}_i(\lk_{\lk_\Delta(\{x\})}(\tau;K)=0$ for all
$\tau\in\lk_\Delta\{x\}$ with $\#\tau\ge t-1$, and all
$i<d-1-\#\tau-1$. By induction hypothesis $\lk_\Delta\{x\}$ is ${\rm
CM}_{t-1}$ for all $\{x\}\in\Delta$. Now by Lemma \ref{link-link} we
are done.

\end{proof}

We state a result due to Munkres \cite[Corollary 3.4]{Mu} which
states that Cohen-Macaulayness is a topological property.

\begin{thm}
Let $\Delta$ be a pure simplicial complex of dimension $(d-1)$. Then
the following are equivalent:

\begin{itemize}

\item[(i)] $\Delta$ is Cohen-Macaulay over $K$.

\item[(ii)] $\tilde{H}_i(|\Delta|;K)=0=H_i(|\Delta|,|\Delta|\setminus p;K)$ for all $p\in|\Delta|$ and all $i<d-1$, where $|\Delta|$ is
the geometric realization of $\Delta$.

\end{itemize}

\end{thm}

The following theorem may lead one to believe that the property ${\rm CM}_t$ is also a topological invariant.

\begin{thm}\label{Munkres}
Let $\Delta$ be a pure simplicial complex of dimension $(d-1)$. Then
the following are equivalent:

\begin{itemize}

\item[(i)] $\Delta$ is ${\rm CM}_t$ over $K$;

\item[(ii)]  $H_i(|\Delta|,|\Delta|\setminus p;K)=0$ for all $p\in|\Delta|\setminus|\Delta_{t-2}|$ and all
$i<d-1$,
where $\Delta_{t-2}$ is the $(t-2)$-skeleton of $\Delta$ and
$|\Delta_{t-2}|$ is induced from a fixed geometric realization of
$\Delta$.

\end{itemize}

\end{thm}

\begin{proof} First note that by Theorem \ref{Reisner}, $\Delta$ is ${\rm CM}_t$ if and only if $\tilde{H}_i(\lk_\Delta(\sigma);K)=0$
for all $\sigma\in\Delta$ with $\#\sigma\ge t$ and all
$i<d-\#\sigma-1$. Now by \cite[Lemma 3.3]{Mu}, for any interior
point $p$ of $\sigma$ we have
$$H_i(|\Delta|,|\Delta|\setminus
p;K)\cong\tilde{H}_{i-\#\sigma}(\lk_\Delta(\sigma);K).$$ Therefore,
$H_i(|\Delta|,|\Delta|\setminus p;K)=0$ for any $\sigma\in\Delta$
with $\#\sigma\ge t$ and any interior point of $\sigma$, and, any
$i<d-1$ if and only if $\tilde{H}_i(\lk_\Delta(\sigma);K)=0$ for all
$\sigma\in\Delta$ with $\#\sigma\ge t$ and $i<d-\#\sigma-1$. But the
set of such points is precisely $|\Delta|\setminus|\Delta_{t-2}|$
when some geometric realization is fixed.

\end{proof}

Let $\Delta$ and $\Delta'$ be a two simplicial complex whose vertex
sets are disjoint. The simplicial join $\Delta *\Delta'$ is defined
to be the simplicial complex whose faces are of the form
$\sigma\cup\sigma'$ where $\sigma\in\Delta$ and $\sigma'\in\Delta'$.

The algebraic and combinatorial properties of the simplicial join
$\Delta *\Delta'$ through the properties of $\Delta$ and $\Delta'$
have been studied by a number of authors (see \cite{BWW},
\cite{Fro}, \cite{PB}, and \cite{AW}). For instance, in~\cite{Fro},
Fr\"oberg used the (graded) $K$-algebra isomorphism $K[\Delta
*\Delta']\simeq K[\Delta]\otimes_K K[\Delta']$, and proved that the
tensor product of two graded $K$-algebras is Cohen-Macaulay (resp.
Gorenstein) if and only if both of them are Cohen-Macaulay (resp.
Gorenstein). One can see that the simplicial join of the
triangulation of a cylinder (which is Buchsbaum ~\cite[Example
II.2.13(i)]{SV}) with a simplicial complex with only one vertex
(which is Cohen-Macaulay~\cite[Example II.2.14(ii)]{SV} and so
Buchsbaum) is not Buchsbaum. In \cite{STY} it is shown that $\Delta
*\Delta'$ is Buchsbaum (over $K$) if and only if $\Delta$ and
$\Delta'$ are Cohen-Macaulay (over $K$). Therefore, it is natural to
ask about $\Delta$ and $\Delta'$ when $\Delta*\Delta'$ is ${\rm
CM}_t$. At present these authors do not know the answer.

\section{The $k$-${\rm CM}_t$ simplicial complexes}

In this section $k$-${\rm CM}_t$ complexes are introduced and some
of their basic properties are given. We show that a complex is
$k$-${\rm CM}_t$ if and only if the links of its nonempty faces are
$k$-${\rm CM}_{t-1}$ ( see Proposition \ref{link1}). The main result
of this section is Theorem \ref{delta1} which states that certain
subcomplex of a ${\rm CM}_t$ complex is $2$-${\rm CM}_t$. This leads
to prove that for an integer $s\le d$, the $(d-s-1)$-skeleton of a
$(d-1)$-dimensional $k$-${\rm CM}_t$ complex is $(k+s)$-${\rm CM}_t$
(see Corollary \ref{3sec3}).

\begin{defn}

Let $K$ be a field. For positive integer $k$ and non-negative
integer $t$, a simplicial complex $\Delta$ with vertex set $V$ is
called $k$-${\rm CM}_t$ of dimension $r$ over $K$ if for any subset
$W$ of $V$ (including $\varnothing$) with $\#W<k$,
$\Delta_{V\setminus W}$ is ${\rm CM}_t$ of dimension $r$ over $K$.
The complex $\Delta$ is $k$-${\rm CM}_t$ over $K$ if $\Delta$ is
$k$-${\rm CM}_t$ of some dimension $r$ over $K$.

\end{defn}
\noindent Note that for any $\ell\le k$, $k$-${\rm CM}_t$ implies
$\ell$-${\rm CM}_t$. In particular, any $k$-${\rm CM}_t$ is ${\rm
CM}_t$.

\vspace{.1in}

In the rest of this paper we will often need the following lemma \cite[Lemma 2.3]{Mi}.

\begin{lem}\label{link}

Let $\Delta$ be a simplicial complex with vertex set $V$. Let
$W\subseteq V$ and let $\sigma$ be a face in $\Delta$. If
$W\cap\sigma=\varnothing$, then $\lk_{\Delta_{V\setminus
W}}(\sigma)=\lk_\Delta(\sigma)_{V\setminus W}$.

\end{lem}

\begin{rem}

One may call $\Delta$ a $(k$-${\rm CM})_t$ complex if for all
$\sigma\in \Delta$ with $\#\sigma \ge t$, $\lk_\Delta(\sigma)$ is
$k$-Cohen-Macaulay. But this is the same as $k$-${\rm CM}_t$
property because both properties require that for $W\subset V$ with
$\#W<k$, $\lk_{\Delta_{V\setminus
W}}(\sigma)=\lk_\Delta(\sigma)_{V\setminus W}$ is Cohen-Macaulay.

\end{rem}

\begin{lem}\label{link2}
Let $\Delta$ be a $k$-${\rm CM}_t$ complex and let $\sigma\in\Delta$ be an arbitrary face with $\#\sigma=s$. Then $\lk_\Delta(\sigma)$ is $k$-${\rm CM}_{t-s}$.

\end{lem}

\begin{proof}

Let $V_1$ be the vertex set of $\lk_\Delta(\sigma)$ and consider
$W\subset V_1$ with $\#W<k$. We need to show that,
$(\lk_\Delta(\sigma))_{V_1\setminus W}$ is ${\rm CM}_{t-s}$. Observe
that since $\sigma \cap W = \varnothing$,
$\lk_\Delta(\sigma)_{V_1\setminus W}=\lk_\Delta(\sigma)_{V\setminus
W}= \lk_{\Delta_{V\setminus W}}(\sigma)$. Put $\Gamma =
\lk_{\Delta_{V\setminus W}}(\sigma)$ and let $\tau\in\Gamma$ with
$\#\tau\ge t-s$. Then $\#(\sigma\cup\tau) \ge t$ and
$\lk_\Gamma(\tau)= \lk_{\Delta_{V\setminus W}}(\sigma\cup \tau)$,
which is Cohen-Macaulay by assumption.
\end{proof}

\begin{cor}

Let $\Delta$ be a $k$-Buchsbaum ($k$-${\rm CM}_2$) complex and let $\sigma\in\Delta$ be a non-empty face. Then $\lk_\Delta(\sigma)$ is $k$-Cohen-Macaulay (resp. $k$-Buchsbaum).

\end{cor}

\begin{prop}\label{link1}

Let $\Delta$ be a pure complex of dimension $(d-1)$ with vertex set $V$. Then for all positive integers $k$ and $t$ the following are equivalent:

\begin{itemize}

\item[(i)] $\Delta$ is $k$-${\rm CM}_t$.

\item[(ii)] For any non-empty face $\sigma$ in $\Delta$, $\lk_\Delta(\sigma)$ is a $k$-${\rm CM}_{t-1}$.

\end{itemize}

\end{prop}

\begin{proof}
(i)$\Rightarrow$ (ii): Use Lemma \ref{link2}.

(ii)$\Rightarrow$ (i): For any subset $W$ of $V$ with $\#W<k$, we
need to show that $\Delta_{V\setminus W}$ is ${\rm CM}_t$ of
dimension $d-1$. Let $\sigma\in\Delta_{V\setminus W}$ with
$\#\sigma\ge t$. Then $\lk_{\Delta_{V\setminus
W}}(\sigma)=(\lk_\Delta(\sigma))_{V\setminus W}$. Since
$\lk_\Delta(\sigma)$ is a $k$-${\rm CM}_{t-1}$ we have that
$\lk_{\Delta_{V\setminus W}}(\sigma)$ is Cohen-Macaulay.

Now we show that $\Delta_{V\setminus W}$ is pure of dimension
$(d-1)$. Let $\tau$ be an arbitrary facet in $\Delta_{V\setminus
W}$. Since $\lk_\Delta(\tau)$ is a $k$-${\rm CM}_{t-1}$ complex, we
have
$$\dim(\lk_\Delta(\tau)_{V\setminus W})=\dim(\lk_\Delta(\tau)).$$ On
the other hand since $\Delta$ is pure, we have
$\dim(\lk_\Delta(\tau))=d-\#\tau-1.$ In addition,
$$\dim(\lk_\Delta(\tau)_{V\setminus W})=\dim(\lk_{\Delta_{V\setminus
W}}(\tau))=\dim(\{\varnothing\})=-1.$$ Therefore, we have
$\dim(\tau)=d-1$.

\end{proof}

\begin{cor} (see \cite[Lemma 4.2]{Mi}) Let $\Delta$ be a pure complex of dimension $(d-1)$ with vertex set
$V$. Then for all positive integers $k$ the following are
equivalent:

\begin{itemize}

\item[(i)] $\Delta$ is $k$-Buchsbaum.

\item[(ii)] For any non-empty face $\sigma$ in $\Delta$, $\lk_\Delta(\sigma)$ is a $k$-Cohen-Macaulay complex.

\end{itemize}

\end{cor}

Now we are ready to give one of the main results of this paper which
generalizes results due to Hibi \cite{Hi} and Miyazaki \cite{Mi}.

Let $\Delta$ a simplicial complex and let
$\sigma_1,\cdots,\sigma_\ell$ be faces of $\Delta$ such that
\begin{itemize}

\item[(i)] $\sigma_i\cup\sigma_j\notin\Delta$ if $i\neq j$.

\item[(ii)] If $\Delta_1=\{\tau\in\Delta|\tau\nsupseteq\sigma_i\,\,\mbox{for all $i$}\,\}$ then $\dim\Delta_i<\dim\Delta$.

\end{itemize}

In \cite{Hi} Hibi showed that $\Delta_1$ is 2-Cohen-Macaulay of
dimension $(\dim\Delta-1)$ provided that $\lk_\Delta(\sigma_i)$ is
2-Cohen-Macaulay for all $i$. In \cite{Mi} Miyazaki extended this
result for Buchsbaumness by showing that if $\Delta$ is a Buchsbaum
complex of dimension $d-1$, and $\lk_\Delta(\sigma_i)$ is
2-Cohen-Macaulay for all $i$, then $\Delta_1$ is 2-Buchsbaum. Now it
is natural to ask whether the similar result is valid for $CM_t$
complexes. In the following result we give an affirmative answer to
this question.

\begin{thm}\label{delta1}

Let $\Delta$ be a ${\rm CM}_t$ complex and let
$\sigma_1,\cdots,\sigma_\ell$ be faces of $\Delta$ satisfying the
above conditions (i) and (ii). If $\lk_\Delta(\sigma_i)$ is 2-${\rm
CM}_{t-1}$ for all $i$, then $\Delta_1$ is 2-${\rm CM}_t$ complex of
dimension $(\dim\Delta-1)$.

\end{thm}

\begin{proof}

We use induction on $t$. If $t=0,1$ the assertion hold by \cite{Hi}
and \cite[Theorem 7.4]{Mi}. Assume that the assertion holds for
$t-1$. By Lemma \ref{link1} we need to show that $\Delta_1$ is pure
and for any non-empty face $\tau$ in $\Delta_1$,
$\lk_{\Delta_1}(\tau)$ is 2-${\rm CM}_{t-1}$. By \cite[Lemma
7.2]{Mi}, $\Delta_1$ is pure. Let $\tau$ be a non-empty face in
$\Delta_1$. We may reorder $\sigma_i$'s such that
$\sigma_i\cup\tau\in\Delta$ if and only if $i\le s$. Then

\[ \begin{array}{rl}
\lk_{\Delta_1}(\tau)= & \{\sigma\in\Delta|\sigma\cup\tau\in\Delta_1, \sigma\cap\tau=\varnothing\}\\
= & \{\sigma\in\Delta|\sigma\cup\tau\in\Delta, \sigma\cap\tau=\varnothing, \sigma\cup\tau\nsupseteq\sigma_i (1\le i\le \ell)\}\\
= &  \{\sigma\in\Delta|\sigma\cup\tau\in\Delta, \sigma\cap\tau=\varnothing, \sigma\nsupseteq\tau_i (1\le i\le s)\}
\end{array}\]
where $\tau_i=\sigma_i-\tau$ for $1\le i\le s$. Thus if we put $\Gamma = \lk_\Delta(\tau)$ then
$$\lk_{\Delta_1}(\tau) =  \{\sigma\in\Gamma|\sigma\nsupseteq\tau_i (1\le i\le s)\}.$$
On the other hand,
$$\lk_{\Gamma}(\tau_i) = \lk_{\Delta}(\tau\cup \tau_i) = \lk_{\lk_{\Delta}(\sigma_i)} (\tau-\sigma_i).$$
By assumption $\lk_{\Delta}(\sigma_i)$ is 2-${\rm CM}_{t-1}$. Then by Lemma \ref{link2}, $\lk_{\lk_{\Delta}(\sigma_i)}(\tau-\sigma_i)$ is
2-${\rm CM}_{t-2}$ and hence $\lk_{\Gamma}(\tau_i)$ is 2-${\rm CM}_{t-2}$. Applying the induction hypothesis for $\Gamma$ and
$\tau_1, \cdots, \tau_s$ it follows that $\lk_{\Delta_1}(\tau)$ is 2-${\rm CM}_{t-1}$. Since $\tau$ is an arbitrary non-empty face of $\Delta_1$,
by Lemma \ref{link1}, it follows that $\Delta_1$ is a 2-${\rm CM}_t$
complex of dimension $(\dim\Delta-1)$.
\end{proof}

The condition on $\lk_\Delta(\sigma_i)$ in the above theorem can not
be weakened in the sense that one can not replace ${\rm CM}_{t-1}$
by ${\rm CM}_t$ for these links. This can be seen in the following
example.

\begin{exam} (see \cite[Example 7.5]{Mi}) If $$\Delta_1=<\{1,2\}, \{1,3\}, \{2,3\}, \{1,4\}, \{1,5\}, \{4,5\}>,$$ which has dimension 1,
and $\Delta_2=<\{x,y\}>$, then $\Delta=\Delta_1*\Delta_2$ is Cohen-Macaulay. If we put $\sigma_1=\{x,y\}$, $t=1$, then $\lk_\Delta(\sigma_1)=\Delta_1$
is a 2-Buchsbaum complex and $\Delta\setminus\sigma_1=\Delta_1*<\{x\},\{y\}>$. So $\lk_{\Delta\setminus\sigma_1}(\{x\})=\Delta_1$ is not 2-Cohen-Macaulay
and we see that $\Delta\setminus\sigma_1$ is not 2-Buchsbaum.

\end{exam}

\begin{cor}\label{3sec3}

Let $\Delta$ be a $k$-${\rm CM}_t$ complex of dimension $(d-1)$. If $s\le d$ and $\Delta'$ is the $(d-s-1)$-skeleton of $\Delta$, then $\Delta'$ is $(k+s)$-${\rm CM}_t$.

\end{cor}

\begin{proof}

We may assume $s=1$. Let $V$ be the vertex set of $\Delta$ and $W$ a
subset of $V$ such that $0<\#W<k+1$. If we take $x\in W$ and put
$W'=W\setminus\{x\}$, then $\Delta_{V\setminus W'}$ is ${\rm CM}_t$
of dimension $(d-1)$ by assumption. On the other hand since
$$\Delta'_{V\setminus W'}=\{\sigma\in\Delta|\dim(\sigma)<d-1, \sigma\cap W'=\varnothing\}$$
and this is equal to the $(d-2)$-skeleton of $\Delta_{V\setminus
W'}$, by Theorem \ref{delta1}, $(\Delta')_{V\setminus W'}$ is
2-${\rm CM_t}$ of dimension $(d-2)$. So $(\Delta')_{V\setminus
W'-\{x\}}=(\Delta')_{V\setminus W}$ is a ${\rm CM}_t$ complex of
dimension $(d-2)$.

\end{proof}

\begin{rem}

The above corollary generalizes a result of Terai and Hibi \cite{TH}
(see also \cite{Ba}) which states that the 1-skeleton of a
simplicial $(d-1)$-sphere with $d\ge 2$ is $d$-connected
(topological). This is just due to the fact that a simplicial
$(d-1)$-sphere is 2-Cohen-Macaulay and $(d-1)$-Cohen-Macaulayness
implies $(d-1)$-connectedness. This corollary also generalizes a
result of Hibi \cite{Hi} (see \cite[Introduction]{Mi}) which says
that if $\Delta$ is a Cohen-Macaulay complex of dimension $d-1$,
then the $(d-2)$-skeleton of $\Delta$ is 2-Cohen-Macaulay.

\end{rem}

\begin{exam}

If  $\Gamma$ is a finite union of $(d-1)$-simplicies where any two
of them intersect in a face of dimension at most $t-2$ and $\Lambda$
is the $(d-2)$-skeleton of $\Gamma$, then $\Lambda$ is 2-${\rm
CM}_t$. If at least two facets in $\Gamma$ intersect in a
$(t-2)$-dimensional face, then $\Lambda$ is not 2-${\rm CM}_{t-1}$.

\end{exam}

\section*{Acknowledgments}
This paper was carried out during Yassemi and Zaare-Nahandi's visit
of Institut de Math\'{e}matiques de Jussieu, Universit\'{e} Pierre
et Marie Curie (Paris 6). They would like to thank the authorities
of the Institut de Math\'{e}matiques de Jussieu, specially Professor
Chardin and Professor Waldschmidt for their hospitality during their
stay in this institute. The authors are grateful to Volkmar Welker
for suggesting several useful comments.

\end{document}